\UseRawInputEncoding
\documentclass[12pt]{amsart}

\usepackage[utf8]{inputenc}
\usepackage{amsmath,amsthm,amssymb,amsfonts,mathtools}
\usepackage{hyperref}
\usepackage{mathrsfs}
\usepackage{enumitem}
\usepackage{tikz-cd}
\usepackage[margin=1.6in]{geometry} 

\usepackage{xcolor}     
\hypersetup{
  colorlinks=true,
  linkcolor=blue!50!black,
  citecolor=blue!50!black,
  urlcolor=blue!50!black
}

\newtheorem{thm}{Theorem}[section]

\newtheorem{prop}[thm]{Proposition}
\newtheorem{lem}[thm]{Lemma}
\newtheorem{cor}[thm]{Corollary}
\newtheorem{rem}[thm]{Remark}
\newtheorem{mydef}[thm]{Definition}

\newcommand{\vol}{\operatorname{vol}}

\newcommand{\Bl}{\mathrm{Bl}}
\DeclareMathOperator{\ord}{ord}
\DeclareMathOperator{\ddc}{dd^c}

\title[cscK metrics on 
birational models of projective varieties]{CscK metrics on 
birational models  
of 
projective varieties}

\author{Zakarias Sj\"ostr\"om Dyrefelt}

\address{Zakarias Sj\"ostr\"om Dyrefelt \\ Institut for Matematik, Aarhus University, Ny Munkegade 118, 8000, Aarhus C, Denmark. }
\email{dyrefelt@math.au.dk}

\begin{document}

\begin{abstract}

We prove that every complex projective variety is birational to a smooth projective manifold admitting a constant scalar curvature K\"ahler (cscK) metric. For any smooth projective variety, a birational cscK model is obtained by resolving the codimension two base locus of a general Lefschetz pencil in a sufficiently positive linear system. The cscK polarization is given explicitly, producing a cscK model also when the initial variety is unstable. 
In dimension two this proves the folklore conjecture that the blowup of any complex projective surface in enough points admits cscK metrics. 




\end{abstract}

\maketitle

\section{Introduction}

\noindent A central problem in K\"ahler geometry is to understand the existence
of canonical metrics on complex manifolds and varieties.  In
particular, constant scalar curvature K\"ahler metrics have been
studied extensively since Calabi's foundational work, culminating
recently in proofs of the uniform Yau--Tian--Donaldson
correspondence \cite{Li,DarvasZhang,BoucksomJonssonYTD,Trusiani}.

A natural question is how flexible cscK existence becomes when one is
allowed to vary the polarized variety.  It is known to propagate under
suitable deformations in polarized families and under suitable
birational modifications of manifolds which already carry canonical
metrics. We ask whether birational modification can instead guarantee
existence without any metric or stability assumption on the original
variety.

\begin{thm} \label{cor main intro}
Every complex projective variety is birational to a smooth projective
variety which admits a cscK metric.
\end{thm}

 More precisely, suppose \(X\) is a complex projective variety. Then there exists a projective
birational morphism
\[
    \mu:\widetilde X\longrightarrow X
\]
from a smooth projective variety \(\widetilde X\) carrying a constant
scalar curvature K\"ahler metric in a rational K\"ahler class. 

Our contributions here pertain to the study of smooth compact projective varieties, and the general result of Theorem \ref{cor main intro} then follows by resolution of singularities \cite{Hironaka}. Once a smooth model is given, we shall see in Theorem \ref{thm main intro} that the morphism $\mu$ may be chosen as a single well-chosen blowup of a codimension two smooth subscheme $Z \subset X$. Combined with Boucksom-Hisamoto-Jonsson \cite{BHJ2} an analogous blowup statement follows also on the algebro-geometric side for uniform K-stability.


In dimension two Theorem~\ref{cor main intro} may be viewed as a
complex-projective analogue of a stabilization theorem of Taubes, who proved that, for every closed oriented smooth four-manifold 
\(M\), the
connected sum
\(
    M\#_N\underline{\mathbb{CP}}^{\,2}
\)
admits an anti-self-dual conformal structure once \(N\) is
sufficiently large \cite{Taubes}.  
A point-blowup formulation for
extremal K\"ahler metrics is considered a folklore open question commonly attributed to Donaldson (see e.g. \cite{DervanSektnan}) and related
cscK formulations were recorded by Tipler and Sz\'ekelyhidi
\cite{Tipler,SzekelyhidiExtremalMetrics}. In complex dimension two, this circle of
questions was pursued in the scalar-flat K\"ahler setting by LeBrun,
LeBrun--Singer and Kim--LeBrun--Pontecorvo
\cite{LeBrunRuled,LeBrunSinger,KimLeBrunPontecorvo}. Here LeBrun--Singer presented a conjecture on blowups of ruled surfaces, which was settled in \cite{KimLeBrunPontecorvo}.

A related series of works produced beautiful results concerning point-blowups of cscK manifolds, showing that cscK and extremal
metrics persist under suitable birational modifications \cite{ArezzoPacard,ArezzoPacardII,ArezzoPacardSinger,SzekelyhidiBlowup,SzekelyhidiBlowupII,SeyyedaliSzekelyhidi,DervanSektnan}. The central method here is gluing, and the results are perturbative in nature, constructing cscK metrics on point blowups of manifolds which already carry such metrics.  

Recent work has moreover studied how canonical metrics, and related algebro-geometric stability notions, vary
in families. For example, Donaldson \cite{DonaldsonFamilies} proved Zariski openness
of the K\"ahler--Einstein locus in families of Fano manifolds with
discrete automorphism groups, and there have been a number of results on variation of uniform K-polystability in families (see e.g. \cite{BlumLiu, BlumLiuXu}). More recently, Dervan \cite{DervanVeryGeneral} proved a striking result that in a
flat family of smooth polarized varieties with finite automorphism
groups, the cscK locus is very general, although it may be empty. 
The above work provides strong tools to 
propagate existence from known fibers, but does not in itself
guarantee a cscK member in an arbitrary family.

Theorem~\ref{cor main intro} is complementary to both
directions as it does not propagate a previously known cscK metric, but instead produces a cscK model even when the initial manifold is unstable, 
following the philosophy of Taubes, LeBrun-Singer's Conjecture 1, and Donaldson (see also \cite{Tipler, SzekelyhidiExtremalMetrics}). 
Our second main result is a more precise version of Theorem \ref{cor main intro} when starting from a \emph{smooth} complex projective variety. We then prove
that the algebraic operation of resolving the base locus of a \emph{sufficiently positive} Lefschetz pencil, in fact also produces a K\"ahler manifold with cscK metrics in a carefully chosen K\"ahler class.

\begin{thm}\label{thm main intro}
Let \(X\) be a smooth complex projective variety of dimension
\(n\geq2\), and let \(H\) be a very ample line bundle on \(X\) such
that
\(
    K_X+H
\)
is ample. Let
\(
    |W|\subset |H|
\)
be any general Lefschetz pencil with its smooth codimension-two base locus
\(Z\subset X\), and let
\[
    \pi:\widetilde X:=\Bl_ZX\longrightarrow X,
    \qquad
    f:\widetilde X\longrightarrow\mathbb P^1
\]
be the associated blowup and Lefschetz fibration. Let \(F := f^*c_1(\mathcal O_{\mathbb{P}^1}(1))\) be the
common divisor class of the fibers of \(f\), 
and denote by 
$\Theta$
the relative canonical class $c_1(K_{\widetilde X}/\mathbb{P}^1)$.
Then \(\Theta\) is a K\"ahler class and $\widetilde X$ admits a cscK metric in 
$$
\Omega_\epsilon := F + \epsilon \Theta 
$$
for all rational $\epsilon > 0$ sufficiently small.
\end{thm}
%
%
\noindent In particular, this solves the point-blowup conjecture on surfaces, and a codimension two version in arbitrary dimensions. 

Theorem \ref{thm main intro} is a more precise version of Theorem \ref{cor main intro} for smooth varieties, showing that the birational morphism $\mu$ can then be taken to be the blowup $\pi$ of the base locus of any sufficiently positive  Lefschetz pencil. The construction is rather flexible: start from any ample line bundle $A$, take a sufficiently large multiple $H = kA$, $k \gg 1$ of $A$. Then the assumptions of Theorem \ref{thm main intro} will be satisfied, with ampleness of $\Theta$ guaranteed by the positivity hypotheses on $H$, and a general Lefschetz pencil can be chosen in $|H|$, whose base locus provides the concrete blowup.

The only remaining non-effective part of the above statement is the size of ``small enough'' positive $\epsilon$. This would be rectified by an effective version of Hattori's adiabatic theorem for the delta invariant, which remains an interesting question for future work. 

\subsection{The surface case, Fano varieties, moduli spaces} 
Given Theorem \ref{thm main intro} there are a number of further consequences that follow directly. We here remark on some of them: 
\newline
\indent \textbf{a)}
When \(X\) is a surface, Theorem \ref{thm main intro} provides a concrete blowup construction that works, in particular, for all complex projective surfaces. The centre of the Lefschetz blowup is a
finite set of distinct points.  More precisely, if \(A\) is a very
ample line bundle on \(X\), then \(K_X+4A\) is ample, and a general
pencil in \(|4A|\) has \(16A^2\) base points.  The simultaneous blowup
of these points therefore carries cscK metrics in the classes $\Omega_\epsilon$ for $0<\epsilon\ll 1$.
This gives a uniform and classification-free version of the
point-blowup statement for projective surfaces, together with an
explicit configuration and an effective bound on the number of
points.  The precise formulation is given in
Section~\ref{sec:applications}.

\textbf{b)} The construction can be run starting from any very ample line bundle $H$ such that $K_X + H$ is ample, leaving plenty of room for choices. In the case of Fano varieties, a particularly natural choice: If \(X\) is smooth Fano and \(m\geq2\) is such that
\(-mK_X\) is very ample, then Theorem~\ref{thm main intro}
applies to a general pencil in \(|-mK_X|\).  In fixed dimension, the
integer \(m\) and the numerical size of the centre may be chosen
uniformly, by boundedness of smooth Fano varieties.

\textbf{c)} Finally, we note that since the results apply to arbitrary complex projective varieties, we may in particular also apply them to study the geometry of moduli spaces which are known to be projective. For exmaple, since Fano K-moduli spaces are projective, every irreducible
component of their reduction admits a smooth projective birational
model carrying a cscK metric, opening the door to applying metric methods to their study.  If such a component is smooth, its
ample CM class may itself be used to define the Lefschetz pencil, so
that the resulting cscK polarization is explicitly related to the CM
polarization.  

\textbf{d)} Section~\ref{sec:applications} also contains a complete
calculation for projective space. We plan to develop these applications further in a second version of this preprint.

\subsection{Comments on the proof} 
\label{subsec:intro-strategy}

We finally remark on the proof. It does \emph{not} use any type of gluing, rather we argue that after applying the procedure in Theorem \ref{thm main intro}, the Mabuchi K-energy functional is coercive for the stated polarization, and invoke Chen-Cheng \cite{ChenCheng} to deduce the existence of a cscK metric (or Boucksom-Hisamoto-Jonsson \cite{BHJ2} to deduce uniform K-stability). Note that the construction implicitly forces, after the prescribed Lefschetz blowup, that the connected component of the automorphism group is a complex torus, as in \cite{SD5}.

The essential difficulty in this approach to the blowup problem is to simultaneously control the energy and
entropy terms in the Chen--Tian decomposition of the Mabuchi
functional. First, already on surfaces the naive approach trying to make the energy part coercive after enough well-chosen blowups, can be shown to always fail when the canonical bundle of the initial surface is \emph{not} pseudoeffective, see \cite{JSD1}. Moreover, it appears very challenging to choose point blowup centres such that the interaction between the energy part and the entropy part is favourable.

The main realization in our proof is that any sufficiently positive Lefschetz pencil
provides the required
correlation between the centre and the polarization. The Lefschetz blowup construction has several advantages. First, the polarization is an adiabatic class $F + \epsilon[K_{Y/\mathbb{P}^1}]$ where $F$ is naturally chosen as a divisor class, i.e. $f^*c_1(\mathcal O_{\mathbb P^1}(1))$. Second, the entropy contribution can be estimated by a direct application of a theorem of Hattori \cite{Hattori}. Third, we observe that the energy contribution, measured via a suitable slope stability threshold defined based on a result of Gao Chen \cite{GaoChen}, can be computed explicitly for certain fibrations over curves,including the resulting Lefschetz fibration after resolving the base locus of the pencil.

The final proof is quite brief. Its brevity reflects the geometric choice: once the centre, fibration
and polarization are identified as a single package, the argument
reduces to an exact slope calculation and an adiabatic entropy limit.
\\

\paragraph{\textbf{Outline of the paper.}}
Section~2 recalls the required background on coercivity, energy and
entropy thresholds, Lefschetz pencils, and Hattori's adiabatic theorem.
In Section~3 we identify the optimally destabilizing subvarieties for
the slope threshold of the fibrations used here.  
Section~4 proves
the main theorems \ref{cor main intro} and \ref{thm main intro}. Section~5 treats the surface case, Fano varieties, Fano
K-moduli spaces, and includes a fully explicit calculation for
projective space.

\subsection*{Acknowledgements.} 
The author would like to thank Claudio Arezzo for bringing this beautiful problem to his attention during his postdoctoral stay at ICTP in Trieste. This work was supported by a Villum Young Investigator Grant from the Villum Foundation, project no. 60786.

\bigskip

\section{Preliminaries}\label{sec:preliminaries}

\noindent Throughout the paper $Y$ denotes a smooth complex projective variety of dimension $n \geq 2$. In the main proof $Y$ denotes the Lefschetz blowup of an initial variety $X$. In general preliminary statements $Y$ is an arbitrary smooth complex projective variety. We use additive notation for line bundles and their first
Chern classes. All intersection
products are numerical.

\subsection{cscK metrics and coercivity}

\noindent Let $Y$ be a compact K\"ahler manifold and $L$ an ample line bundle on $Y$. Write $\alpha := c_1(L) \in H^{1,1}(Y,\mathbb{Z})$ for its associated first Chern class and set $V := \alpha^n$ for the K\"ahler volume. 
We say that $Y$ \emph{admits a constant scalar curvature (cscK) metric in $\alpha$} if there exists a K\"ahler form $\omega \in \alpha$ whose scalar curvature is constant, i.e. $\omega$ solves the non-linear fourth order elliptic \emph{cscK equation}
$$
S(\omega) = \overline{S},
$$
where
$$
S(\omega) = \mathrm{Tr}_{\omega}\mathrm{Ric}(\omega) , \; \; \; \overline{S} = -n\frac{c_1(K_Y).\alpha^{n-1}}{\alpha^n} \in \mathbb{R}. 
$$
Existence of cscK metrics is known to be equivalent to coercivity (modulo automorphisms) of the Mabuchi functional, defined on the space of K\"ahler potentials of $Y$. To define this and fix notation, let $\omega \in \alpha$ be a reference K\"ahler form, and write
$$
  \mathcal H_\omega
  :=
  \{\varphi\in C^\infty(Y,\mathbb R):
  \omega_\varphi:=\omega+\ddc\varphi>0\}
$$
The Mabuchi functional \(\mathrm M=\mathrm M_\omega\), normalized by
\(\mathrm M(0)=0\), is characterized by
\[
  \frac{d}{dt}\mathrm M(\varphi_t)
  =
  -\frac1V\int_Y
  \dot\varphi_t
  \bigl(S(\omega_{\varphi_t})-\overline S\bigr)
  \omega_{\varphi_t}^{\,n};
\]
in particular, its critical points $\varphi$ correspond precisely to the cscK metrics $\omega_\varphi$ in
\(\alpha\).

We moreover recall the Chen--Tian decomposition in the normalization used
throughout. Write $\rho := -\mathrm{Ric}(\omega) \in c_1(K_Y)$ for the negative of the Ricci curvature form, 
and denote by
\[
  \mathrm E(\varphi)
  :=
  \frac{1}{(n+1)V}
  \sum_{j=0}^{n}
 \int_Y
  \varphi\,
  \omega_\varphi^{\,j}\wedge\omega^{n-j}
\]
the Monge--Amp\`ere energy. We also recall Aubin's \(I\)- and \(J\)-functionals,
\[
    \mathrm I_\omega(\varphi)
    :=
    \frac{1}{V}
   \int_Y
    \varphi\bigl(\omega^n-\omega_\varphi^n\bigr),
    \qquad
    \mathrm J_\omega(\varphi)
    :=
    \frac{1}{V}\int_Y\varphi\,\omega^n-E(\varphi),
\]
The entropy and the remaining energy part
are respectively
\[
  \mathrm H(\varphi)
  :=
  \frac1V\int_Y
  \log\frac{\omega_\varphi^{\,n}}{\omega^n}\,
  \omega_\varphi^{\,n},
\]
and 
\begin{equation} \label{eq:energy functional}
  \mathrm E_\rho(\varphi)
  :=
  \frac1V\sum_{j=0}^{n-1}
 \int_Y
  \varphi\,\rho\wedge
  \omega_\varphi^{\,j}\wedge\omega^{n-1-j}
  +\overline S\mathrm E(\varphi).
\end{equation}
Thus
\[
  \mathrm M=\mathrm E_\rho+\mathrm H.
\]
Both terms are invariant under the addition of constants, and
we say that $M$ is \emph{coercive on $\mathcal{H}_{\omega}$} if there are constants $C,D > 0$ such that
$$
\mathrm M \geq C \mathrm E_{\omega} - D
$$
on $\mathcal{H}_{\omega}$. Here $\mathrm E_\omega := \mathrm I- \mathrm J$, a notation motivated by e.g. \cite[Proposition 9]{SD4}. In order for the full equivalence between existence of cscK metrics and coercivity to hold, one must take into account the action of the automorphism group $\mathrm{Aut}_0(Y)$, but that is not needed for the following statement, which is all we will use here. 

\begin{thm}[\cite{ChenCheng}]
Suppose that $\mathrm M$ is coercive on $\mathcal{H}_\omega$. Then $Y$ admits a cscK metric in $[\omega] = c_1(L)$.
\end{thm}

\noindent When no confusion can arise, we shall simply say that $M$ is coercive in $c_1(L)$.

\medskip

\subsection{Energy and entropy stability thresholds and sufficient existence criteria for cscK metrics} \label{subsec:thresholds}

We will make use of a \emph{sufficient} criterion for existence of cscK metrics using coercivity thresholds as follows. Coercivity is equivalent to
$$
 \sup \{C \in \mathbb R : \exists D_{C} \in \mathbb R :  \mathrm M \geq C \mathrm E_{\omega} - D_C \; \text{on} \; \mathcal H_{\omega} \} > 0.
$$
We can moreover introduce the following suprema associated with each term in the Chen-Tian decomposition, which were studied in \cite{SD4} and \cite{Zhang} respectively. The \emph{energy threshold} is
\newline
$$
 \Gamma^{\mathrm pp}_{\rho}(L) := \sup \{C \in \mathbb R : \exists D_{C} \in \mathbb R :  \mathrm E_\rho \geq C \mathrm E_{\omega} - D_C \ \text{on} \ \mathcal H_{\omega} \} > -\infty
$$
\newline
and the \emph{analytic delta invariant} is defined as
\newline
$$
\delta^{\mathrm an}(L) := \sup \{C \in \mathbb R : \exists D_{C} \in \mathbb R :  \mathrm H \geq C \mathrm E_{\omega} - D_C \; \text{on} \; \mathcal H_{\omega} \} > 0.
$$
\newline
The above quantities are purely cohomological, i.e. do not depend on representatives within the cohomology classes involved (see \cite[Remark 3]{SD4} and \cite{Zhang}). 
We choose to emphasize the line bundles $L$ and $K_Y$ in our notation, as this is natural for their use in the proof of our main results. 
Because $E_\rho\geq a \mathrm E_{\omega}-A$ and $H\geq b \mathrm E_{\omega}-B$ implies $$M\geq(a+b) \mathrm E_{\omega}-(A+B),$$ we moreover have the inequality
\[
    \Gamma^{\mathrm{pp}}_\rho(L)
    +\delta^{\mathrm{an}}(L)
    \leq
    \sup\left\{
        C\in\mathbb R:
        \exists D_C\in\mathbb R,\ 
        M\geq C \mathrm E_{\omega}-D_C
        \text{ on }\mathcal H_\omega
    \right\}.
\]
Hence a sufficient condition for coercivity is that the left-hand
side above is positive.
\begin{prop}[Sufficient existence criterion for cscK] \label{prop:energy-entropy}
Suppose that 
\newline
$$
\Gamma^{\mathrm pp}_{\rho}(L) + \delta^{\mathrm an}(L) > 0.
$$
\newline
Then the Mabuchi functional is coercive. In particular, $Y$ admits
a cscK metric in $c_1(L)$.
\end{prop}

\noindent Both terms involved have algebraic expressions, useful to compute or estimate the thresholds. 

\subsubsection{A slope criterion for positivity of the energy threshold} 
From this point on, all quantities under consideration are
cohomological, and we freely use the same notation for a line bundle
and its first Chern class. More generally, $L$ will denote a K\"ahler
class and $\beta\in H^{1,1}(Y,\mathbb R)$ a real $(1,1)$-class.
For such a class $\beta$, the notation $\mathrm E_\beta$ and
$\Gamma^{\mathrm{pp}}_\beta(L)$ refers to the corresponding twisted
energy and its coercivity threshold, defined using any smooth
representative of $\beta$.
More precisely, if $\theta\in\beta$ is a smooth representative and
\[
\underline\beta
:=
n\frac{\beta\cdot L^{n-1}}{L^n},
\]
then $E_\beta$ is obtained from \eqref{eq:energy functional} by replacing $\rho$ by
$\theta$ and $\bar S$ by $-\underline\beta$.

To study the energy part of the Mabuchi functional, and its non-archimedean counterpart, consider the following stability threshold inspired by the study of J-stability:

\begin{mydef}[Slope threshold]\label{def:pp-threshold}
Let $L$ be a K\"ahler class on $Y$ and let
$\beta\in H^{1,1}(Y,\mathbb R)$.
For an irreducible subvariety
$V\subset Y$ of dimension $1\leq p\leq n-1$, set
\[
  \mathcal E_{\beta,L}(V)
  :=
  \frac{
  \left(
  n\dfrac{\beta\cdot L^{n-1}}{L^n}L-p\beta
  \right)\cdot L^{p-1}\cdot V
  }{(n-p)L^p\cdot V}.
\]
Define also
\[
\mathcal E_\beta(L)
:=
\inf_{1\le p\le n-1}
\inf_{\substack{V\subset Y\ {\rm irreducible}\\ \dim V=p}}
\mathcal E_{\beta,L}(V).
\]
which we shall refer to as the \emph{slope} or \emph{slope threshold}.
\end{mydef}

\noindent The normalization is chosen so that $\mathcal E_{L,L}(V)=1$. Moreover,

\begin{lem} \label{lem magic J-formula} For every $a\in\mathbb R$,
\begin{equation}\label{eq:affine-energy}
  \Gamma^{\mathrm{pp}}_{\beta+aL}(L)
  =\Gamma^{\mathrm{pp}}_\beta(L)+a, \; \; \text{and} \; \; \mathcal E_{\beta + aL}(L) = \mathcal E_{\beta}(L) + a.
\end{equation}
\end{lem}

\begin{rem}[Example: Surface form of the threshold]\label{rem:surface-threshold} 
If $\dim Y=2$, then the only proper positive-dimensional subvarieties are
curves, and Kleiman's criterion gives
\[
  t_{\mathrm{nef}}(Y,L)
  :=\inf\{t\in\mathbb R:tL-K_Y\ \text{is nef}\}
  =\sup_{C\subset Y}\frac{K_Y\cdot C}{L\cdot C}.
\]
Consequently,
\[
  \mathcal E_{K_Y}(L)
  =2\frac{K_Y\cdot L}{L^2}-t_{\mathrm{nef}}(Y,L),
\]
and in \cite{JSD1} we prove that ampleness of
\[
  Q_L:=2\frac{K_Y\cdot L}{L^2}L-K_Y,
\]
is equivalent to coercivity of the energy part $\mathrm E_\rho$ of the Mabuchi functional. It is worth noting the identity
\[
  Q_L^2=K_Y^2.
\]
For the blowup of $r$ distinct points, the last square equals
$K_X^2-r$. This provides a scale-invariant obstruction behind the difficulty of
making the canonical energy nonnegative by point blowups alone, and from this it follows, see \cite{JSD1}, that for any compact K\"ahler surface with non-psef canonical bundle $K_X$, the thresholds $\Gamma^{\mathrm{pp}}_{K_Y}(L)$ and $\mathcal E_{K_Y}(L)$ are always strictly negative.
\end{rem}

\noindent We will rely on the following main theorem of Gao Chen \cite{GaoChen} on J-stability, which crucially assumes the auxiliary class $\beta$ is a K\"ahler class.

\begin{thm}[\cite{GaoChen}] \label{thm:GaoChen}
Let $L$ and $\beta$ be K\"ahler classes on $Y$. Then
$
  \mathcal E_\beta(L) > 0
$
if and only if $\Gamma^{\mathrm{pp}}_\beta(L) > 0$. 
\end{thm}

\subsubsection{Algebraic vs analytic delta invariant} \label{subsec:algdelta}

Let \((Y,L)\) be a smooth polarized variety of dimension \(n\).
A \emph{prime divisor over \(Y\)} is a prime divisor \(G\) on a
smooth projective variety \(Y'\) equipped with a projective birational
morphism
\(
    \mu:Y'\longrightarrow Y.
\)
We set
\[
    A_Y(G)
    :=
    1+\ord_G\bigl(K_{Y'}-\mu^*K_Y\bigr)
\]
and
\[
    S_L(G)
    :=
    \frac{1}{L^n}
    \int_0^{+\infty}
    \vol_{Y'}\bigl(\mu^*L-xG\bigr)\,dx.
\]
These quantities depend only on the divisorial valuation
\(\ord_G\), and not on the choice of the smooth model on which \(G\)
appears.

The algebraic delta invariant, or stability threshold, is
\[
    \delta(Y,L)
    :=
    \inf_{G\ \mathrm{prime\ over}\ Y}
    \frac{A_Y(G)}{S_L(G)}.
\]
Equivalently, the infimum may be taken over all nontrivial divisorial
valuations \(v\) on \(Y\):
\[
    \delta(Y,L)
    =
    \inf_v\frac{A_Y(v)}{S_L(v)}.
\]
This is the valuative formulation of the stability threshold of
Fujita--Odaka and Blum--Jonsson
\cite{FujitaOdakaDelta,BlumJonssonThresholds}.

A theorem of K.~Zhang shows that the two
thresholds are equal.

\begin{thm}[\cite{Zhang}]
  Let \(Y\) be a smooth projective variety and let \(L\) be an ample
  \(\mathbb R\)-line bundle. Then
  \[
    \delta^{\mathrm{an}}\bigl(Y,c_1(L)\bigr)
    =
    \delta(Y,L).
  \]
\end{thm}

\noindent In particular, we shall henceforth write simply
\(\delta(Y,L)\) for either the algebraic or analytic threshold.

\subsection{Lefschetz pencils and fibrations}
\label{subsec:lefschetz}

Lefschetz fibrations are a holomorphic analogue of Morse functions, and arise from preliminary structures called Lefschetz pencils, see e.g. \cite{SGA7II, VoisinHodgeII} for background. We will restrict our study to Lefschetz pencils and fibrations on smooth complex projective varieties. 
To introduce them, let \(X\) be a smooth complex projective variety of dimension \(n\geq 2\),
and let \( H\) be a very ample line bundle on \(X\). 
A pencil in \(|H|\) is the linear subsystem
\[
    |W|:=\mathbb P(W)\subset|H|
\]
associated with a two-dimensional subspace \(W\subset H^0(X, H)\). Upon
choosing a basis \(s_0,s_1\) of \(W\), its members are
\[
    D_t:=\{t_0s_0+t_1s_1=0\},
    \qquad
    t=[t_0:t_1]\in\mathbb P^1,
\]
and its base locus is
\[
    Z_W:=\operatorname{Bs}|W|
        =\{s_0=s_1=0\}.
\]
The pencil defines a rational map
\[
    \phi_W:X\dashrightarrow\mathbb P^1,
    \qquad
    x\longmapsto[s_1(x):-s_0(x)],
\]
whose fiber over \(t\) is \(D_t\setminus Z_W\).

We say that \(|W|\) is a \emph{Lefschetz pencil} if \(Z_W\) is smooth
of codimension two and \(\phi_W\) has nondegenerate critical points
with distinct critical values. By the classical Lefschetz theorem, the
Lefschetz pencils form a nonempty Zariski-open subset
\[
    \mathcal U_{\mathrm{Lef}}
    \subset\operatorname{Gr}(2,H^0(X, H));
\]
see
\cite[Expos\'e~XVII, Th\'eor\`eme~2.5.2 and
Corollaire~3.2.1]{SGA7II}.

Fix \(W\in\mathcal U_{\mathrm{Lef}}\) and put \(Z:=Z_W\). The
incidence variety
\[
    \mathcal Y_W
    :=
    \{(x,t)\in X\times\mathbb P^1:x\in D_t\}
\]
is naturally isomorphic to the blowup of \(X\) along \(Z\), see
\cite[\S2.3.1]{VoisinHodgeII}. We therefore write
\(
    \pi: Y :=\operatorname{Bl}_Z X\longrightarrow X
\)
for the first projection and
\(
    f:Y\longrightarrow\mathbb P^1
\)
for the second. We thus have the following picture.
\begin{figure}[htbp]
\centering
\begin{tikzcd}[row sep=large, column sep=large]
& Y \arrow[dl, "\pi"'] \arrow[dr, "f"] & \\
X & & \mathbb{P}^1
\end{tikzcd}
\caption{The fibered geometry of the regularizing blowup.}
\label{fig:lefschetz_fibration}
\end{figure}
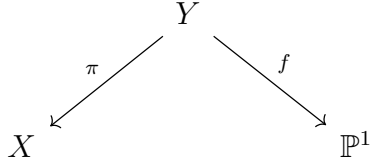
 The map \(f\) is called the \emph{Lefschetz
fibration} associated with the pencil. The fiber
\(Y_t:=f^{-1}(t)\) is the strict transform of \(D_t\), and
\[
    \pi|_{Y_t}:Y_t\xrightarrow{\ \simeq\ }D_t,
\]
since \(Z\subset D_t\) is a Cartier divisor.

Let \(E\subset Y\) denote the exceptional divisor, and let \(F\)
denote the divisor class of a fiber of \(f\). The following summarizes the only properties we need about the Lefschetz fibration and its base locus. Once they are given, the remaining argument is purely variational and cohomological.
\begin{prop}\label{prop:lefschetz-package}
The morphism
\(
    f:Y\longrightarrow\mathbb P^1
\)
is a contraction. Moreover,
\(
    F=\pi^*H-E.
\)
In particular, \(F\) is semiample and nef, and
\(
    F^2=0.
\)
Every fiber \(Y_t\) is reduced and satisfies
\[
    \operatorname{lct}(Y,Y_t)=1.
\]
\end{prop}
\begin{proof}
A smooth member of \(|H|\) is connected, see
\cite[Corollary~III.7.9]{Hartshorne}. Hence the general fiber of
\(f\) is connected, and Stein factorization gives
\(
    f_*\mathcal O_{Y}\simeq\mathcal O_{\mathbb P^1}.
\)
Thus \(f\) is a contraction.
Since the base locus is cut out transversally by \(s_0\) and \(s_1\),
each member \(D_t\) of the pencil contains \(Z\) with multiplicity
one. Its total transform is therefore
\(
    \pi^*D_t=Y_t+E.
\)
As \(D_t\in|H|\) and \(Y_t\) represents \(F\), this gives
\(
    F=\pi^*H-E.
\)
The fibers form a basepoint-free pencil in the class \(F\). Hence
\(F\) is semiample and nef. Moreover, two distinct fibers are
disjoint, and hence
\(
    F^2=0.
\)

By the Lefschetz condition and the holomorphic Morse lemma, every
singular fiber has a unique singular point, near which there are
holomorphic coordinates \(z_1,\ldots,z_n\) on \(Y\) and a local
coordinate \(\tau\) on \(\mathbb P^1\) such that
\(
    \tau\circ f=z_1^2+\cdots+z_n^2.
\)
In particular, every fiber is reduced. 

Finally,
\[
    \operatorname{lct}_0
    \bigl(z_1^2+\cdots+z_n^2\bigr)
    =
    \min\left\{1,\frac n2\right\}
    =1
\]
for \(n\geq2\), see \cite[Example~1.8]{MustataLCT}. At every smooth
point of \(Y_t\), the local log canonical threshold is also equal to
\(1\), which proves the last assertion.
\end{proof}

\subsection{Hattori's adiabatic theorem for the delta invariant}
\label{subsec:hattori-delta}

The ordinary delta invariant recalled in subsection \ref{subsec:algdelta} is the
special case \(B=0\) of the delta invariant of a polarized log pair. For background on the valuative theory underlying this subsection, we
refer the reader to \cite{BlumJonssonThresholds, Hattori} and references therein.

Let \((Y,B)\) be a projective klt pair and let \(L\) be an ample
rational line bundle on \(Y\).
For a prime divisor \(G\) over \(Y\), choose a log resolution
\(
    \mu:Y'\longrightarrow Y
\)
of \((Y,B)\) on which \(G\) appears, and set
\[
    A_{(Y,B)}(G)
    :=
    1+\ord_G
    \bigl(
        K_{Y'}-\mu^*(K_Y+B)
    \bigr).
\]
The log delta invariant is
\[
    \delta_{(Y,B)}(L)
    :=
    \inf_{G\ \mathrm{prime\ over}\ Y}
    \frac{A_{(Y,B)}(G)}{S_L(G)}.
\]
In particular,
\[
    \delta_{(Y,0)}(L)=\delta(Y,L).
\]

Now let
\(
    f:Y\longrightarrow C
\)
be a projective \emph{contraction} from a smooth projective variety to a
smooth projective curve, i.e. a projective morphism such that $f_* \mathcal O_{Y} = \mathcal O_C$. For \(p\in C\), we denote by \(f^*p\) the
scheme-theoretic fiber and set
\[
    B_f
    :=
    \sum_{p\in C}
    \left(
       1-\operatorname{lct}(Y,f^*p)
    \right)p.
\]
This is the \emph{discriminant boundary} of \(f\).

We use the following smooth special case of Hattori's adiabatic
convergence theorem. The result applies to arbitrary polarized
algebraic fiber spaces over curves. In particular, no klt-triviality
assumption is imposed.

\begin{thm}[Hattori {\cite[Theorem~4.12]{Hattori}}]
\label{thm:hattori}
Let \(A\) be an ample rational class on \(Y\), and let \(L_C\) be an
ample rational class on \(C\). Then
\[
    \lim_{\epsilon\to0^+}
    \delta\bigl(Y,f^*L_C+\epsilon A\bigr)
    =
    \delta_{(C,B_f)}(L_C),
\]
where the limit is taken through positive rational values of
\(\epsilon\).
\end{thm}

The right-hand side is the ordinary log delta invariant of the
polarized pair \((C,B_f,L_C)\). If
\[
    B=\sum_{p\in C}b_p p,
\]
then every divisorial valuation of \(C\) is a positive multiple of a
point valuation, and we have $A_{(C,B)}(\ord_p)=1-b_p$ and
\[
    S_{L_C}(\ord_p)=\frac{\deg L_C}{2}.
\]
Consequently,
\[
    \delta_{(C,B)}(L_C)
    =
    \frac{2}{\deg L_C}
    \min_{p\in C}(1-b_p),
\]
see \cite[Example~2.17]{Hattori}. Equivalently,
Theorem~\ref{thm:hattori} gives
\[
    \lim_{\epsilon\to0^+}
    \delta\bigl(Y,f^*L_C+\epsilon A\bigr)
    =
    \frac{2}{\deg L_C}
    \min_{p\in C}
    \operatorname{lct}(Y,f^*p).
\]

\begin{cor}\label{cor:hattori-lefschetz}
Let
\(
    f:Y\longrightarrow\mathbb P^1
\)
be the Lefschetz fibration constructed in subsection \ref{subsec:lefschetz}, and let
\(F\) denote its fiber class. Then \(B_f=0\), and for every ample
rational class \(A\) on \(Y\),
\[
    \lim_{\epsilon\to0^+}
    \delta\bigl(Y,F+\epsilon A\bigr)
    =
    2.
\]
\end{cor}

\begin{proof}
By Proposition~2.8, every fiber has log-canonical threshold one, and
hence \(B_f=0\). Apply Theorem~\ref{thm:hattori} with $L_C=\mathcal O_{\mathbb P^1}(1)$ and $\deg L_C=1$.
\end{proof}

\medskip

\section{Destabilizing subvarieties for fibrations over curves}

\noindent Let $Y$ be a smooth complex projective variety of dimension $n \geq 2$. Let $L$ be a K\"ahler class,
and $\beta\in H^{1,1}(Y,\mathbb R)$.  

In this section we introduce one of the new elements of our proof: it is possible to exactly compute the slope threshold
$\mathcal E_\beta(L)$ for certain types of fibrations over curves. To do this, we study the notion of \emph{optimally destabilizing subvarieties}, building on ideas from \cite{KSD1,KSD2}. We say that a proper irreducible subvariety $V \subset Y$ of dimension $p \in \{1,2,\dots, n-1\}$ is \emph{optimally destabilizing} if it computes the slope threshold, i.e.
\newline
$$
\mathcal E_\beta(L) = \frac{
  \left(
  n\dfrac{\beta\cdot L^{n-1}}{L^n}L-p\beta
  \right)\cdot L^{p-1}\cdot V
  }{(n-p)L^p\cdot V}.
$$

\medskip

\subsection{Optimally destabilizing subvarieties for the slope threshold } \label{subsec: opt dest subvar}

Inspired by \cite{KSD1} and \cite{KSD2} we consider the class 
$$
D_{Y} := t_{\mathrm{nef}}L - K_{Y},
$$
where
$$
t_{\mathrm{nef}} = \inf\{s\in\mathbb R:sL-K_{Y}\ \text{is nef}\}.
$$
In the case when this class is nef and big, it is known for surfaces (see 
\cite{KSD1,KSD2}) that any optimally destabilizing subvariety must be contained in the non-K\"ahler locus $E_{nK}(D_{Y})$.
%
%
Going beyond the case of surfaces, the situation becomes more complicated, and in general, including many of the most interesting applications, $D_{Y}$ will not be big. Then the non-K\"ahler locus is all of $Y$ and such an approach would not provide any useful information.

In this note we therefore make the following observation, which does not depend on $K_{Y}$ being sufficiently positive:

\begin{prop}[Localization by a mixed-null divisor]
\label{prop:mixed-null-divisor}
Let \(L\) be a K\"ahler class on a smooth projective \(n\)-fold \({Y}\) and let
\(\beta\in H^{1,1}(Y, \mathbb R)\) be a real class. Define
\[
  t:=\inf\{s\in\mathbb R:sL-\beta\ \text{is nef}\},
\]
and
\[
  D:=tL-\beta.
\]
Since the nef cone is closed, $D$ is nef.
Assume moreover that \(D\not\equiv0\) and that there exists a prime divisor
\(W\subset {Y}\) such that
\(
  D\cdot L^{n-2}\cdot W=0.
\)
Then the subvarieties computing
\(\mathcal E_{\beta}(L)\) are precisely the prime divisors
\(V\subset {Y}\) satisfying
\(
  D\cdot L^{n-2}\cdot V=0.
\)
Finally, we then have the explicit formula
\[
  \mathcal E_{\beta}(L)
  =
  t-n\frac{D\cdot L^{n-1}}{L^n}.
\]
\end{prop}

\begin{proof}
For an irreducible \(p\)-dimensional subvariety \(V\subset {Y}\), consider the slope quantities
\[
  r_p(V):=
  \frac{D\cdot L^{p-1}\cdot V}{L^p\cdot V},
  \qquad
  r_n:=
  \frac{D\cdot L^{n-1}}{L^n}.
\]
Since \(D\) is nef, \(r_p(V)\geq0\).  Moreover, recall the shorthand
$$
\mathcal E_{\beta,L}(V) := \frac{
  \left(
  n\dfrac{\beta\cdot L^{n-1}}{L^n}L-p\beta
  \right)\cdot L^{p-1}\cdot V
  }{(n-p)L^p\cdot V}
$$
Since $\beta = tL - D$, one then computes
\[
  \mathcal E_{\beta,L}(V)
  =
  t-\frac{nr_n-pr_p(V)}{n-p}
\]
and here
\[
  \frac{nr_n-pr_p(V)}{n-p}
  \leq
  \frac{nr_n}{n-p}
  \leq nr_n.
\]
Since $D$ is a nonzero nef class and $L$ is K\"ahler, one has $D.L^{n-1} > 0$. Hence 
\(r_n>0\).  Equality in both inequalities holds
if and only if $p=n-1$ and $
  D\cdot L^{n-2}\cdot V=0.
$
The assumed divisor \(W\) shows that equality is attained, and the conclusion follows. 
\end{proof}

\medskip

\subsection{Localization of optimal destabilizers for fibrations over curves}

Fibrations over curves provide a natural source of divisors satisfying
the mixed-null condition in Proposition \ref{prop:mixed-null-divisor}, and hence of optimally
destabilizing divisors. 
A prime divisor \(V\subset Y\) is called vertical if \(f(V)\) is a
point, and horizontal if \(f(V)=C\).

\begin{cor}[Fibrations over curves]
\label{cor:mixed-null-fibration}
Let
\(
f:Y\longrightarrow C
\)
be a surjective morphism with connected fibers onto a smooth
projective curve. Let $L$ be a K\"ahler class and let
$\beta\in H^{1,1}(Y,\mathbb R)$. Suppose that the corresponding
nef-threshold class is of the form
\[
D=f^*A_C
\]
for a K\"ahler class $A_C$ on $C$. Then the subvarieties computing
\(\mathcal E_\beta(L)\) are precisely the vertical prime
divisors. In particular, every smooth fiber computes
\(\mathcal E_\beta(L)\).
\end{cor}

\begin{proof}
A vertical prime divisor \(V\) is an irreducible component of a fiber,
and hence
\(
    D|_V=0.
\)
Thus
\(
    D\cdot L^{n-2}\cdot V=0.
\)
On the other hand, if \(V\) is horizontal, then positivity of \(A_C\)
and \(L\), together with the projection formula, gives
\[
    D\cdot L^{n-2}\cdot V>0.
\]
The first assertion therefore follows from
Proposition~\ref{prop:mixed-null-divisor}.

Finally, a smooth fiber is connected by assumption. Since a smooth
variety has disjoint irreducible components, a connected smooth fiber
is irreducible, and hence is itself a vertical prime divisor.
\end{proof}

\noindent In the following section the strategy is therefore to show that we can choose a general and sufficiently positive Lefschetz fibration $f: Y \rightarrow \mathbb P^1$ such that $D$ is a positive multiple of the fiber class $f^*(c_1(\mathcal O_{\mathbb P^1}(1)))$. Then Corollary \ref{cor:mixed-null-fibration} applies and $\mathcal E_\beta (L)$ can be computed exactly by the formula from Proposition \ref{prop:mixed-null-divisor}.

\medskip

\section{Existence of cscK metrics on Lefschetz blowups}
\label{sec:main-proof}

\noindent We now prove Theorems \ref{cor main intro} and \ref{thm main intro} by applying the results of the preceding sections to the Lefschetz
fibration associated with a sufficiently positive pencil. To lighten the
notation, throughout this section we write
\[
    \pi:Y:=\Bl_ZX\longrightarrow X
\]
for the Lefschetz blowup denoted by \(\widetilde X\) in the
introduction. Throughout this section we identify divisors and line bundles with
their first Chern classes whenever they occur in cohomological
expressions. In particular, all expressions involving the real parameter
$\varepsilon$ are understood in $H^{1,1}(Y,\mathbb R)$.

\subsection{Statement of the main result}
\label{subsec:main-statement}

Let \(X\) be a smooth complex projective variety of dimension
\(n\geq2\), and let \(A\) be an ample line bundle on \(X\). Choose an integer
\(N \geq 1\) such that
\(
    H_N:=NA
\)
is very ample and
\[
    B_N:=K_X+H_N
\]
is ample. Such integers exist, since every sufficiently large $N$ has these properties. Choose moreover a general Lefschetz pencil
\[
    \{D_t\}_{t\in\mathbb P^1}\subset |H_N|
\]
with smooth codimension-two base locus \(Z\subset X\), and let
\[
    \pi:Y:=\Bl_ZX\longrightarrow X,
    \qquad
    f:Y\longrightarrow\mathbb P^1
\]
be the associated blowup and Lefschetz fibration. Denote by \(E\)
the exceptional divisor and by \(F := f^*(c_1(\mathcal O_{\mathbb P^1}(1)))\) the common divisor class of the
fibers of \(f\). By Proposition~\ref{prop:lefschetz-package},
\[
    F=\pi^*H_N-E.
\]
Set
\[
\Theta_N:=c_1(K_{Y/\mathbb P^1})=K_Y+2F,
\qquad
L_{N,\varepsilon}:=F+\varepsilon\Theta_N.
\]
The dependence of \(Y,Z,E,f\), and \(F\) on \(N\) will be left implicit, and suppressed
from the notation. We now aim to prove: 

\begin{thm}\label{thm:main}
In the setting above, the relative canonical class
\(\Theta_N\) is K\"ahler on $Y$. Moreover, there exists
\(\epsilon_0>0\) such that, for every rational number
\(
    0<\epsilon<\epsilon_0,
\)
the Mabuchi functional is coercive in the class
\(L_{N,\epsilon}\). In particular,
\(L_{N,\epsilon}\) contains a cscK metric.
\end{thm}

\begin{rem}[Automorphisms]
No discreteness assumption on \(\mathrm{Aut}(Y)\) is required. The argument
proves ordinary, absolute coercivity of the Mabuchi functional, which
is stronger than coercivity modulo automorphisms. Thus any
holomorphic direction incompatible with absolute coercivity is
excluded by the conclusion itself. 
The construction implicitly forces, after the prescribed blowup, that the connected component of the automorphism group is a complex torus, as in \cite{SD5}.
\end{rem}

\subsection{$L_{N,\epsilon}$ is K\"ahler}
\label{subsec:kahler-polarization}

Since \(Z\) is the transverse common zero locus of two sections
defining the pencil, every member \(D_t\in|H_N|\) contains \(Z\)
with multiplicity one. Hence its total transform under the blowup is
\[
    \pi^*D_t=Y_t+E,
\]
where \(Y_t\) is the strict transform of \(D_t\). Since \(Y_t\) is a
fiber of \(f\), it represents \(F\), and therefore
\[
    F=\pi^*H_N-E.
\]
Since the blowup centre has codimension two,
\(
    K_Y=\pi^*K_X+E.
\)
Combining the two identities gives
\begin{equation} \label{eq:KplusF}
    K_Y+F=\pi^*(K_X+H_N)=\pi^*B_N.
\end{equation}

\begin{lem}\label{lem:relative-canonical}
The relative canonical class satisfies
\[
    \Theta_N
    =
    K_{Y/\mathbb P^1}
    =
    \pi^*B_N+F
\]
and is K\"ahler. Hence, the following classes are K\"ahler for every \(\epsilon>0\). 
\begin{enumerate}
    \item $L_{N,\epsilon}
    =
    F+\epsilon\Theta_N$
    \item $K_Y+L_{N,\epsilon}
    =
    \pi^*B_N+\epsilon\Theta_N$
\end{enumerate}
\end{lem}

\begin{proof}
Since
\(
    K_{\mathbb P^1}
    =
    -2c_1\bigl(\mathcal O_{\mathbb P^1}(1)\bigr),
\)
and \(F=f^*c_1(\mathcal O_{\mathbb P^1}(1))\), one has
\[
    K_{Y/\mathbb P^1}
    =
    K_Y-f^*K_{\mathbb P^1}
    =
    K_Y+2F.
\]
Moreover,
\[
    K_Y+2F
    =
    (K_Y+F)+F
    =
    \pi^*B_N+F.
\]
To see that this class is K\"ahler, consider the incidence embedding 
\(
    \iota=(\pi,f):Y\hookrightarrow X\times\mathbb P^1
\)
Then
\[
    \Theta_N
    =
    \iota^*
    \left(
        \mathrm{pr}_X^*B_N
        +
        \mathrm{pr}_{\mathbb P^1}^*
        c_1\bigl(\mathcal O_{\mathbb P^1}(1)\bigr)
    \right).
\]
The class in parentheses is K\"ahler on
\(X\times\mathbb P^1\), and hence its restriction to \(Y\) is
K\"ahler.
Finally, \(F\) is nef, so
\(
    F+\epsilon\Theta_N
\)
is K\"ahler for every \(\epsilon>0\). Since \(\pi^*B_N\) is nef,
the same argument shows that
\(
    \pi^*B_N+\epsilon\Theta_N
\)
is K\"ahler.
\end{proof}

\subsection{The sufficient condition for existence of cscK metrics}

We aim to prove that our Lefschetz blowup satisfies the following sufficient existence criterion for cscK metrics, which is an immediate consequence of the 
entropy--energy criterion from section \ref{subsec:thresholds}. 

\begin{prop}[Entropy--energy existence criterion for cscK metrics]\label{prop:goal-criterion}
Let $L$ be the class of an ample $\mathbb Q$-line bundle on $Y$.
Assume that
\begin{enumerate}
    \item $K_Y+L$ is K\"ahler
    \item $\delta(Y,L)>1$
    \item $ \Gamma^{\mathrm{pp}}_{K_Y+L}(L)>0.$
\end{enumerate}
Then the Mabuchi functional is coercive for $L$, i.e. on the associated space of K\"ahler potentials for any K\"ahler form representing the class. 
\end{prop}

\begin{proof}
By linearity of the energy threshold,
\[
    \Gamma^{\mathrm{pp}}_{K_Y+L}(L)
    =
    1+\Gamma^{\mathrm{pp}}_{K_Y}(L).
\]
The assumptions therefore imply
\[
    \delta(Y,L)+
    \Gamma^{\mathrm{pp}}_{K_Y}(L)>0.
\]
The conclusion follows from
Proposition~\ref{prop:energy-entropy}.
\end{proof}

\begin{rem}
After rescaling, the alleged cscK class on the blowup $Y = \mathrm{Bl}_Z X$ can be written as 
\begin{equation}\label{eq:normalized-blowup-class}
  L_{N,\epsilon}'
  :=\frac{L_{N,\epsilon}}{N(1+2\epsilon)}
  =\pi^*\left(
     A+\frac{\epsilon}{N(1+2\epsilon)}K_X
   \right)
   -\frac{1+\epsilon}{N(1+2\epsilon)}[E].
\end{equation}
Thus the downstairs class is a small perturbation of $c_1(A)$ in the direction of the canonical class $K_X$, and allows a comparison with the standard K\"ahler classes $\pi^*\alpha - \epsilon[E]$ appearing in the blowup literature.
\end{rem}

\medskip

\subsection{Computation of the slope threshold for Lefschetz pencils}
\label{subsec:slope-computation}

Let \(F=c_1(f^*\mathcal O_{\mathbf P^1}(1))\) be the fiber class introduced above. Fix \(\epsilon>0\), and write for short
\[
    L:=L_{N,\epsilon}
      =F+\epsilon\Theta_N.
\]
By Lemma~\ref{lem:relative-canonical}, both \(L\) and
\(K_Y+L\) are K\"ahler. Recall from subsection \ref{subsec: opt dest subvar} that the slope
threshold associated with the pair \((K_Y+L,L)\) is governed by the
nef threshold
\[
    t
    :=
    \inf\bigl\{
        s\in\mathbb R:
        sL-(K_Y+L)\ \text{is nef}
    \bigr\}
\]
and the corresponding nef class
\[
    D:=tL-(K_Y+L).
\]
Since $\Theta_N = K_Y + 2F$ we have
\[
    L=F+\epsilon(K_Y+2F),
\]
and
\begin{equation} \label{eq:cancellation}
    K_Y+L
    = L + \epsilon^{-1}(L-F) - 2F = 
    (1+\epsilon^{-1})L
    -
    (\epsilon^{-1}+2)F.
\end{equation}
As \(F\) is nef, this shows that
\[
    t\leq1+\epsilon^{-1}.
\]
Conversely, let \(Y_u\) be a smooth fiber of \(f\). Since
\(F|_{Y_u}=0\), for every \(s<1+\epsilon^{-1}\) one has
\[
\left(
    sL-(K_Y+L)
\right)\big|_{Y_u}
=
\left(
    s-1-\epsilon^{-1}
\right)L\big|_{Y_u},
\]
which is not nef. Hence the nef threshold precisely equals
\[
    t=1+\epsilon^{-1}
\]
and cancellation using \eqref{eq:cancellation} gives
\[
    D
    =
    (\epsilon^{-1}+2)F
    =
    f^*\!\left(
        (\epsilon^{-1}+2)c_1
        \bigl(\mathcal O_{\mathbb P^1}(1)\bigr)
    \right).
\]
Corollary~\ref{cor:mixed-null-fibration} therefore applies:
every smooth fiber computes the slope threshold, and Proposition
\ref{prop:mixed-null-divisor} gives
\begin{equation} \label{eq:slope}
    \mathcal E_{K_Y+L}(L)
    =
    1+\epsilon^{-1}
    -
    (\epsilon^{-1}+2)n
    \frac{F\cdot L^{n-1}}{L^n}.
\end{equation}
\noindent In order to estimate this quantity, we rewrite the exact formula for the slope threshold as follows.

\begin{prop}[Exact slope]\label{prop:exact-canonical-energy}
Use the shorthand $L=L_{N,\epsilon}$ and $B_N = K_X + H_N$. Then
\begin{equation}\label{eq:one-plus-delta-pp}
 \mathcal E_{K_Y + L}(L) = 1 + \mathcal E_{K_Y}(L)
  =\frac{B_N^n+\epsilon\Theta_N^n}
  {nH_N\cdot B_N^{n-1}+\epsilon\Theta_N^n}
  >0.
\end{equation}
Hence 
$$
\Gamma^{\mathrm{pp}}_{K_Y + L}(L) > 0.
$$
\end{prop}

\begin{proof}
Since $F^2=0$ and $L=F+\epsilon\Theta_N$ we compute
\[
  L^n
  =\epsilon^{n-1}
   \left(nF\cdot\Theta_N^{n-1}
   +\epsilon\Theta_N^n\right), \qquad F\cdot L^{n-1}
  =\epsilon^{n-1}F\cdot\Theta_N^{n-1}.
\]
Subtracting \(1\) from the formula \eqref{eq:slope} for
\(\mathcal E_{K_Y+L}(L)\), and using \(F^2=0\), gives
\[
 \mathcal E_{K_Y}(L) = \epsilon^{-1}
    -
    (\epsilon^{-1}+2)n
    \frac{F\cdot L^{n-1}}{L^n} = 
  \frac{\Theta_N^n-2nF\cdot\Theta_N^{n-1}}
  {nF\cdot\Theta_N^{n-1}+\epsilon\Theta_N^n}.
\]
Because $\Theta_N=K_Y+2F$ and $F^2=0$,
\[
  F\cdot\Theta_N^{n-1}=F\cdot K_Y^{n-1}
\]
and
\[
  \Theta_N^n-2nF\cdot\Theta_N^{n-1}=K_Y^n.
\]
The expression therefore simplifies to
\begin{equation}\label{eq:exact-delta-pp}
  \mathcal E_{K_Y}(L)
  =\frac{K_Y^n}
  {nF\cdot K_Y^{n-1}+\epsilon\Theta_N^n}.
\end{equation}
Next, again using $F^2=0$,
\[
  K_Y^n+nF\cdot K_Y^{n-1}=(K_Y+F)^n.
\]
By equation \eqref{eq:KplusF}, $K_Y+F=\pi^*B_N$, and
\[
  F\cdot K_Y^{n-1}
  =F\cdot\pi^*B_N^{n-1}
  =H_N\cdot B_N^{n-1}.
\]
Indeed, for the last equality, if \(Y_t\) is a fiber corresponding to a member
\(D_t\in|H_N|\), then \(\pi|_{Y_t}:Y_t\to D_t\) is an isomorphism.
Therefore
\[
    F\cdot\pi^*B_N^{n-1}
    =
    \int_{Y_t}(\pi^*B_N|_{Y_t})^{n-1}
    =
    \int_{D_t}(B_N|_{D_t})^{n-1}
    =
    H_N\cdot B_N^{n-1}.
\]
Adding one to \eqref{eq:exact-delta-pp} therefore yields
\eqref{eq:one-plus-delta-pp}: 
$$
\mathcal E_{K_Y + L}(L) = 1 + \mathcal E_{K_Y}(L) = 1 + \frac{K_Y^n}
  {nF\cdot K_Y^{n-1}+\epsilon\Theta_N^n}
= \frac{B_N^n+\epsilon\Theta_N^n}
  {nH_N\cdot B_N^{n-1}+\epsilon\Theta_N^n}.
$$
Its numerator is positive
because \(B_N\) and \(\Theta_N\) are K\"ahler. Its denominator is
positive because \(H_N\) is ample and \(B_N\) is K\"ahler, so
\[
    H_N\cdot B_N^{n-1}>0,
\]
while \(\Theta_N^n>0\). 
The final assertion follows from Theorem \ref{thm:GaoChen}.
\end{proof}

\begin{rem}[Asymptotic energy]\label{rem:asymptotic-energy}
The exact positivity in \eqref{eq:one-plus-delta-pp} is sufficient for the
proof, but it may be of independent interest to note that an asymptotic statement is available: if we first let
$\epsilon\to0$ and then $N\to\infty$, then since
\[
  K_Y=\pi^*B_N-F,
  \qquad
  F\cdot K_Y^{n-1}=H_N\cdot B_N^{n-1},
\]
one obtains
\[
  \lim_{\epsilon\to0^+}
  \mathcal E_{K_Y}(L_{N,\epsilon})
  =-\frac{n-1}{n}
   +\frac{K_X\cdot(K_X+H_N)^{n-1}}
   {nH_N\cdot(K_X+H_N)^{n-1}},
\]
and hence
\[
  \lim_{N\to\infty}\lim_{\epsilon\to0^+}
  \mathcal E_{K_Y}(L_{N,\epsilon})
  =-\frac{n-1}{n}.
\]
\end{rem}

\medskip

\subsection{Adiabatic limit of the delta invariant}

The preceding results explain the energy localization for any fibration
over a curve whose nef threshold class is pulled back from the base.  The
Lefschetz construction supplies three additional properties simultaneously:
it produces such a fibration from every smooth projective manifold, the
direction of perturbation \(K_Y+2F\) is K\"ahler, and the singular fibers have
log-canonical threshold one.  In this paragraph we observe that the last property allows us to compute the full adiabatic delta invariant by a direct application of a theorem of Hattori \cite{Hattori}.

\begin{lem}\label{lem:adiabatic-delta}
For the Lefschetz fibration \(f:Y\to\mathbb P^1\),
\[
    \lim_{\epsilon\to0^+}
    \delta(Y,L_{N,\epsilon})=2.
\]
In particular,
\[
    \delta(Y,L_{N,\epsilon})>1
\]
for every sufficiently small positive rational \(\epsilon\).
\end{lem}

\begin{proof}
This is Corollary~\ref{cor:hattori-lefschetz}, applied with the ample
class \(A=\Theta_N\). We repeat the brief argument for convenience of the reader. Proposition~\ref{prop:lefschetz-package} gives
$\operatorname{lct}_Y(Y_t)=1$ for every fiber, hence $B_f=0$.
Applying Theorem~\ref{thm:hattori} with $C=\mathbb P^1$, $ L_C=c_1\bigl(\mathcal O_{\mathbb P^1}(1)\bigr)$, and $A=\Theta_N$
yields
\[
  \lim_{\epsilon\to0^+}\delta(Y,L_{N,\epsilon})
  =\delta\bigl(\mathbb P^1,\mathcal O_{\mathbb P^1}(1)\bigr).
\]
Every prime divisor over \(\mathbb P^1\) is represented by a point
\(q\in\mathbb P^1\), so it suffices to consider the valuation
\(\operatorname{ord}_q\). Its log discrepancy is
\(A_{\mathbb P^1}(\operatorname{ord}_q)=1\). Moreover,
\[
  \operatorname{vol}\bigl(\mathcal O_{\mathbb P^1}(1)-xq\bigr)
  =\max\{1-x,0\},
\]
and hence
\[
  S_{\mathcal O(1)}(\operatorname{ord}_q)
  =
  \int_0^\infty
  \operatorname{vol}\bigl(\mathcal O_{\mathbb P^1}(1)-xq\bigr)\,dx
  =
  \int_0^1(1-x)\,dx
  =
  \frac12.
\]
Therefore
\[
  \delta\bigl(\mathbb P^1,\mathcal O_{\mathbb P^1}(1)\bigr)=2
\]

and the conclusion follows. 
\end{proof}

\noindent Thus the fibration structure allows us to localize and control the energy along the blowup, while the Lefschetz condition supplies
a lower bound on the analytic delta invariant, which is the entropy threshold.

\medskip

\subsection{Energy-entropy compensation and proof of Theorem \ref{thm main intro}}
We now conclude the proof of our main theorem \ref{thm main intro} by comparing estimates of the respective energy and entropy thresholds.


\begin{proof}[Proof of Theorem~\ref{thm:main}]
Let $X$ be a smooth complex projective variety, let $A$ be an ample line bundle and fix \(N\) 
so that $H_N := NA$ is very ample and $K_X + H_N$ is ample. By Lemma~\ref{lem:adiabatic-delta}, for every sufficiently small positive rational number $\epsilon$, one has
\[
\delta(Y,L_{N,\epsilon}) > 1.
\] 
For every $\epsilon > 0$, both \(L_{N,\epsilon}\) and
\[
    K_Y+L_{N,\epsilon}
    =
    \pi^*B_N+\epsilon\Theta_N
\]
are K\"ahler, and Proposition~\ref{prop:exact-canonical-energy} gives
\[
    \Gamma^{\mathrm{pp}}_{K_Y+L_{N,\epsilon}}
    (L_{N,\epsilon})>0,
\]
and hence
\[
    \Gamma^{\mathrm{pp}}_{K_Y}
    (L_{N,\epsilon}) = \Gamma^{\mathrm{pp}}_{K_Y+L_{N,\epsilon}}
    (L_{N,\epsilon}) - 1 > -1.
\]
Combining the two, Proposition~\ref{prop:goal-criterion} implies coercivity of the
Mabuchi functional. The cscK conclusion follows from \cite{ChenCheng}. 
\end{proof}

\noindent Taking $A=H$ and $N=1$ proves Theorem 1.2.

\begin{rem}
Note that we are obliged to use Gao Chen's theorem \cite{GaoChen} for a K\"ahler class $K_Y + L$, at the cost of worsening the coercivity estimate. For surfaces this extra step is not necessary, due to the improved ampleness criterion for non-psef classes, see \cite{JSD1}.
\end{rem}

\medskip

\subsection{Proof of Theorem \ref{cor main intro}}

\begin{cor}\label{cor:birational-model}
Every complex projective variety is birational to a smooth projective
variety which admits a cscK metric.
\end{cor}

\begin{proof}
We may suppose that
$\dim X\ge2$. First apply a projective resolution of singularities, due to Hironaka \cite{Hironaka}. Then apply
Theorem~\ref{thm:main} to the resulting smooth projective variety.
\end{proof}

\medskip

\section{Applications} \label{sec:applications}

\subsection{Point blowups on projective surfaces}
\label{sec:surfaces}

\noindent In the special case of projective surfaces, the codimension two base locus in Theorem \ref{thm main intro} is a union of distinct points. It is natural to ask about the numerical quantity
\newline
\[
  b_{\mathrm{cscK}}(X)
  :=
  \min\left\{
      r\geq0:
      \begin{array}{c}
      \text{there exist distinct points }p_1,\ldots,p_r\in X\\
      \text{such that }\Bl_{\{p_1,\ldots,p_r\}}X
      \text{ admits a cscK metric} \\ \text{in some K\"ahler class}
      \end{array}
  \right\}
\]
\newline
where the points $p_1,\ldots,p_r\in X$ are required to be distinct, and a manifold is said to be cscK if it admits a cscK metric in some K\"ahler class. We use the convention that the minimum is \(+\infty\) if no such
collection exists. 
Thus, the conjectural picture of \cite{Taubes, LeBrunSinger, Tipler, SzekelyhidiExtremalMetrics} and Donaldson amounts to proving $ b_{\mathrm{cscK}}(X)$ is finite. 

The constructive nature of the Lefschetz blowup allows us to provide a (coarse but universal) estimate on how many points are required to blowup, in order to stabilize a projective surface.

\begin{thm}\label{thm:surface-bound}
Let \(A\) be a very ample line bundle on a smooth projective surface
\(X\), and let \(m\geq1\) be such that \(K_X+mA\) is ample. A general
Lefschetz pencil in \(|mA|\) has \(r=m^2A^2\) distinct base points. If
\(\pi:Y\to X\) denotes their blowup, then
\[
    F+\epsilon K_{Y/\mathbb P^1}
\]
contains a cscK metric for every sufficiently small rational
\(\epsilon>0\). In particular,
\[
    b_{\mathrm{cscK}}(X)\leq m^2A^2,
\]
so \(b_{\mathrm{cscK}}(X)\) is finite.
\end{thm}

\begin{proof}
The base locus of a general pencil in \(|mA|\) is the transverse
intersection of two members of the linear system, and hence consists
of \((mA)^2=m^2A^2\) distinct points. The conclusion follows from
Theorem~\ref{thm:main}.
\end{proof}

\begin{cor}\label{cor:universal-surface-bound}
For every very ample line bundle \(A\) on a smooth projective surface
\(X\),
\[
    b_{\mathrm{cscK}}(X)\leq16A^2.
\]
Consequently,
\[
    b_{\mathrm{cscK}}(X)
    \leq16\min_{A\ \mathrm{very\ ample}}A^2.
\]
In particular, the point-blowup conjecture holds for every smooth
projective surface.
\end{cor}

\begin{proof}
By Mori's cone theorem and the length bound for extremal rays
\cite[Theorems~1.13 and~1.24]{KM98}, one has
\[
    \overline{\mathrm{NE}}(X)
    =
    \overline{\mathrm{NE}}(X)_{K_X\geq0}
    +\sum_i R_i,
\]
where each \(K_X\)-negative extremal ray \(R_i\) is generated by a
rational curve \(C_i\) satisfying \(0<-K_X\cdot C_i\leq3\).
Since \(A\) is ample and Cartier, \(A\cdot C_i\geq1\), and therefore
\((K_X+3A)\cdot C_i\geq0\). The class \(K_X+3A\) is also nonnegative
on \(\overline{\mathrm{NE}}(X)_{K_X\geq0}\), and is therefore nef.
Thus
\(K_X+4A=(K_X+3A)+A\) is ample, and the result follows from
Theorem~\ref{thm:surface-bound} with \(m=4\).
\end{proof}

\noindent If \(X\subset\mathbb P^N\) has degree \(d\), taking \(A\) to be the
hyperplane class gives \(b_{\mathrm{cscK}}(X)\leq16d\).

\begin{rem}
The point configuration and its cardinality are effective: one takes
the common zero locus of two general sections of \(4A\).  The only non-effective part of this construction, is the choice of admissible
adiabatic parameter \(\epsilon\). It remains an interesting open question to determine $b_{\mathrm{cscK}}(X)$ exactly.
\end{rem}

\medskip 

\subsection{Lefschetz blowups on Fano manifolds}

The construction takes an especially intrinsic form for smooth Fano
varieties. Since \(-K_X\) is ample, every sufficiently divisible
pluri-anticanonical system satisfies the positivity hypothesis of
Theorem \ref{thm main intro}. Thus both the Lefschetz centre and the resulting cscK
polarization are obtained directly from the anticanonical geometry of
\(X\), and in fixed dimension, Birkar's boundedness theorem makes the
construction uniform.

\begin{cor}\label{cor:fano}
Let \(X\) be a smooth Fano variety of dimension \(n\geq2\), and let
\(m\geq2\) be such that \(-mK_X\) is very ample. Choose a general
Lefschetz pencil in \(|-mK_X|\), with base locus \(Z\), and let
\(\pi:Y:=\Bl_ZX\to X\) be the associated blowup. Then, for every
sufficiently small rational \(\epsilon>0\), the class
\[
  F+\epsilon K_{Y/\mathbb P^1}
\]
contains a cscK metric. Moreover,
\[
    (-K_X)^{n-2}\cdot Z=m^2(-K_X)^n.
\]
\end{cor}

\begin{proof}
Apply Theorem~\ref{thm:main} with \(A=-K_X\) and \(N=m\), noting that
\(K_X-mK_X=-(m-1)K_X\) is ample. The numerical identity follows since
\(Z\) is the transverse intersection of two members of
\(|-mK_X|\).
\end{proof}

\noindent Writing \(E\) for the exceptional divisor, the resulting polarization is
\[
    F+\epsilon K_{Y/\mathbb P^1}
    =\bigl(m+(2m-1)\epsilon\bigr)\pi^*(-K_X)-(1+\epsilon)E.
\]

\begin{rem}
In this case the exact slope formula of
Proposition~\ref{prop:exact-canonical-energy} gives
\[
  \lim_{\epsilon\to0^+}
  \mathcal E_{K_Y+L_\epsilon}(L_\epsilon)
  =
  \frac{m-1}{nm}>0.
\]
\end{rem}

\begin{cor}\label{cor:fano-uniform}
For every \(n\geq2\), there exist an integer \(m_n\geq2\) and a constant
\(R_n>0\) such that every smooth Fano \(n\)-fold \(X\) admits a cscK
birational model obtained by blowing up the smooth base locus of a
general pencil in
\[
    |-m_nK_X|,
\]
and the center satisfies
\[
    (-K_X)^{n-2}\cdot Z\leq R_n.
\]
\end{cor}
\begin{proof}
By Birkar's solution \cite{Bir21} of the Borisov--Alexeev--Borisov conjecture,
smooth Fano varieties of fixed dimension form a bounded family. Hence, there exist an integer \(m_n\geq2\) and
a constant \(V_n>0\), depending only on \(n\), such that
\(
    -m_nK_X
\)
is very ample and
\(
    (-K_X)^n\leq V_n
\)
for every smooth Fano \(n\)-fold \(X\). Apply
Corollary~\ref{cor:fano} and set
\(
    R_n:=m_n^2V_n.
\)
\end{proof}

\medskip

\subsection{Example: An explicit model on projective space}
\label{subsec:projective-space}
Projective space provides a model in which the centre, polarization,
fibers and slope threshold can all be computed explicitly. 

Let for $n\geq 2$
\[
    X=\mathbb P^n,
    \qquad
    A=\mathcal O_{\mathbb P^n}(1),
    \qquad
    N=n+2.
\]
Then
\[
    H_N=(n+2)A,
    \qquad
    B_N=K_X+H_N=A.
\]
Choose a general Lefschetz pencil of hypersurfaces of degree \(n+2\),
with smooth base locus
\[
    Z=D_0\cap D_\infty
       \subset\mathbb P^n,
\]
a complete intersection of type \((n+2,n+2)\), and let
\[
    \pi:Y:=\Bl_Z\mathbb P^n\longrightarrow\mathbb P^n,
    \qquad
    f:Y\longrightarrow\mathbb P^1
\]
be the associated Lefschetz blowup.  Write
\[
    h:=\pi^*c_1\bigl(\mathcal O_{\mathbb P^n}(1)\bigr)
\]
and let \(E\) denote the exceptional divisor.  The fiber and relative
canonical classes are
\[
    F=(n+2)h-E
\]
and
\[
    \Theta_N
    =
    K_{Y/\mathbb P^1}
    =
    (n+3)h-E
    =
    h+F.
\]
Consequently, Theorem~\ref{thm:main} gives cscK metrics in the classes
\[
    L_\epsilon
    :=
    F+\epsilon\Theta_N
    =
    \bigl(n+2+(n+3)\epsilon\bigr)h
    -(1+\epsilon)E
\]
for all sufficiently small rational \(\epsilon>0\).

The fibers have a particularly natural geometry.  A smooth fiber is
a hypersurface \(D\subset\mathbb P^n\) of degree \(n+2\), and
adjunction gives
\[
    K_D
    =
    \bigl(K_{\mathbb P^n}+D\bigr)|_D
    =
    \mathcal O_D(1).
\]
Thus the Lefschetz fibration is a family of canonically polarized
hypersurfaces, polarized by their canonical bundles.

In this example the slope threshold can also be written completely
explicitly.  The incidence realization identifies \(Y\)
with a hypersurface of bidegree \((n+2,1)\) in
\(\mathbb P^n\times\mathbb P^1\).  If \(h\) and \(q\) denote the
hyperplane classes of the two factors, then on \(Y\)
\[
    F=q,
    \qquad
    \Theta_N=h+q.
\]
Since \(q^2=0\), we obtain
\[
\begin{aligned}
    \Theta_N^n
    &=
    \int_{\mathbb P^n\times\mathbb P^1}
       (h+q)^n\bigl((n+2)h+q\bigr)  \\
    &=
    1+n(n+2)
    =
    (n+1)^2.
\end{aligned}
\]
Proposition~\ref{prop:exact-canonical-energy} therefore specializes to
\[
    \mathcal E_
      {K_{Y}+L_\epsilon}(L_\epsilon)
    =
    \frac{1+(n+1)^2\epsilon}
         {n(n+2)+(n+1)^2\epsilon}
    >0.
\]

After rescaling, the cscK classes take the form
\[
    h-
    \frac{1+\epsilon}
         {n+2+(n+3)\epsilon}E
\]
and approach the exact nef-boundary class
\[
    h-\frac{1}{n+2}E
    =
    \frac{1}{n+2}F
\]
as \(\epsilon\to0^+\).

For \(n=2\), the base locus consists of the sixteen points of
intersection of two general plane quartics.  The fibers are canonical
curves of genus three, and the construction produces cscK metrics on
\[
    \Bl_{\{p_1,\ldots,p_{16}\}}\mathbb P^2
\]
in the classes
\[
    (4+5\epsilon)h
    -(1+\epsilon)\sum_{i=1}^{16}E_i,
\]
with exact slope threshold
\[
    \mathcal E_
      {K_{Y}+L_\epsilon}(L_\epsilon)
    =
    \frac{1+9\epsilon}{8+9\epsilon}.
\]

\subsection{Fano K-moduli spaces}

The projectivity theorem for Fano K-moduli spaces allows us to apply
the main result not only to Fano varieties themselves, but also to the
parameter spaces which classify them \cite{LXZ22}. 

\begin{cor}\label{cor:fano-kmoduli}
Let \(M^{\mathrm{Kps}}_{n,V}\) be the projective K-moduli space of
\(n\)-dimensional K-polystable \(\mathbb Q\)-Fano varieties of
anticanonical volume \(V\).  Every irreducible component $M$ of
\[
    \bigl(M^{\mathrm{Kps}}_{n,V}\bigr)_{\mathrm{red}}
\]
admits a projective birational morphism
\[
    \mu:\widetilde M\longrightarrow M
\]
from a smooth projective variety \(\widetilde M\) carrying a cscK
metric in a rational K\"ahler class.
\end{cor}

\begin{proof}
Every irreducible component \(M\) of
\(
    \bigl(M^{\mathrm{Kps}}_{n,V}\bigr)_{\mathrm{red}}
\)
is projective. If the dimension of $M$ is at least two, the conclusion follows directly
from Corollary~\ref{cor:birational-model}. If \(\dim M=1\), the
normalization of \(M\) is a smooth projective curve and hence carries
a cscK metric. The zero-dimensional case is immediate.
\end{proof}

\begin{rem}\label{rem:cm-polarization}
Suppose that \(M\) is a smooth irreducible component of a Fano
K-moduli space, of dimension at least two, and denote its ample CM
\(\mathbb Q\)-line bundle by \(\Lambda_{\mathrm{CM}}\)
\cite{XZ20,LXZ22}. Choose \(m\) sufficiently large and divisible that
\(
    H:=m\Lambda_{\mathrm{CM}}
\)
is a very ample line bundle and \(K_M+H\) is ample. A general
Lefschetz pencil in \(|H|\) then gives a blowup
\[
    \pi:\widetilde M\longrightarrow M
\]
carrying cscK metrics in the classes
\[
    \pi^*\!\left(
        (1+2\varepsilon)m\Lambda_{\mathrm{CM}}
        +\varepsilon K_M
    \right)
    -(1+\varepsilon)E,
    \qquad
    0<\varepsilon\ll1.
\]
Thus the resulting cscK polarization is explicitly related to the
natural CM polarization on the moduli space.
\end{rem}

\noindent It remains an open question for higher dimensional varieties, whether it suffices to blow up finitely many well-chosen points, as opposed to codimension two smooth subschemes.

\end{document}